\documentclass[a4paper,11pt]{article}

\usepackage[latin1]{inputenc}
\usepackage{amsmath}
\usepackage{amsthm}
\usepackage{amssymb}
\usepackage{mathrsfs}
\usepackage{graphicx}
\usepackage{hyperref}

\usepackage{abstract} 

\newtheorem{thm}{Theorem}
\newtheorem{cor}[thm]{Corollary}
\newtheorem{claim}[thm]{Claim}
\newtheorem{lemma}[thm]{Lemma}
\newtheorem{prop}[thm]{Proposition}

\theoremstyle{definition}
\newtheorem{definition}[thm]{Definition}

\newtheorem{remark}[thm]{Remark}

\newtheorem{question}[thm]{Question}

\title{Acylindrical hyperbolicity from actions on CAT(0) cube complexes: a few criteria}
\date{October 30, 2016}
\author{Anthony Genevois}

\begin{document}

\maketitle

\begin{abstract}
The question which motivates the article is the following: given a group acting on a CAT(0) cube complex, how can we prove that it is acylindrically hyperbolic? Keeping this goal in mind, we show a weak acylindricity of the action on the contact graph associated to a CAT(0) cube complex, and we prove a characterisation of WPD contracting isometries of CAT(0) cube complexes. As a first consequence, we find alternative arguments to show several criteria which were proved by Indira Chatterji and Alexandre Martin. Next, we show that if a group acts essentially and \emph{acylindrically on the hyperplanes} (i.e., the intersection of the stabilisers of two hyperplanes which are sufficiently far away from each other has its cardinality uniformly bounded) on a finite-dimensional CAT(0) cube complex, then it contains a finite-index subgroup which is either acylindrically hyperbolic or cyclic. Finally, we prove a statement about codimension-one subgroups, which implies that, if a group contains a finitely generated codimension-one subgroup which is malnormal and which satisfies the bounded packing property, then this group must contain a finite-index subgroup which is either acylindrically hyperbolic or cyclic.
\end{abstract}

\tableofcontents

\section{Introduction}

A well-known strategy to study groups from a geometric point of view is to find ``nice'' actions on spaces which are ``nonpositively-curved'', or even better, which are ``negatively-curved''. The most iconic illustration of this idea comes from Gromov's seminal paper \cite{GromovHyp} introducing \emph{hyperbolic groups}, i.e., groups acting geometrically on (Gromov-)hyperbolic spaces. A major theme in geometric group theory is to generalise technics and properties available for hyperbolic groups to a wider context. In the last years, one of the most impressive such generalisation was the introduction of \emph{acylindrically hyperbolic groups} \cite{OsinAcyl}, a class of groups unifying several previous generalisations and including many examples of classical groups. We refer to \cite{OsinSurvey} and references therein for more details.

Formally, a group $G$ is \emph{acylindrically hyperbolic} if it admits an action on a (Gromov-)hyperbolic space $X$ which is \emph{non-elementary} (i.e., with a infinite limit set) and \emph{acylindrical}, i.e., for every $d \geq 0$, there exist some $R,N \geq 0$ such that, for all $x,y \in X$,
\begin{center}
$d(x,y) \geq R \Rightarrow \# \{ g \in G \mid d(x,gx),d(y,gy) \leq d \} \leq N$.
\end{center}
In general, it is not an easy task to prove that a given group is acylindrically hyperbolic. The main reason is that, in the definition of an acylindrical action, we are not looking at stabilisers of points, but ``almost stabilisers'': we are looking at the elements of the group moving a point within some distance. This motivates the need of finding useful criteria allowing us to prove that some groups are acylindrically hyperbolic.

It is worth noticing that the previous difficulty disappears for actions on trees, where the acylindricity reduces to the \emph{weak acylindricity}: there exist some constants $R,N \geq 0$ such that, for all $x,y \in X$, 
\begin{center}
$d(x,y) \geq R \Rightarrow |\mathrm{stab}(x) \cap \mathrm{stab}(y)| \leq N$. 
\end{center}
This simple observation turns out to be quite useful. For instance, many groups of interest are shown to be acylindrically hyperbolic in \cite{MinasyanOsin}. 

In this article, our goal is to prove new criteria of acylindrical hyperbolicity for groups acting on \emph{CAT(0) cube complexes}, which are usually considered as generalised trees in higher dimensions. Our motivation for focusing on these spaces is twofold. Firstly, the similarity between the geometries of simplicial trees and CAT(0) cube complexes suggests that studying acylindrical hyperbolicity in cube complexes should be much easier than in other spaces. In order to justify this idea, let us mention that, generalising \cite[Theorem A]{articleMartin}, we proved in \cite[Theorem 8.33]{coningoff} that an action on a hyperbolic CAT(0) cube complex is acylindrical if and only if it is weakly acylindrical, like for trees. Secondly, many groups of interest turn out to act on CAT(0) cube complexes, providing a large and interesting collection of potential applications.

So our context is the following: we have a group acting on a CAT(0) cube complex and we would like to prove that it is acylindrically hyperbolic. In other words, we would like to construct a hyperbolic space on which our group acts nicely. But we already know how to construct a hyperbolic space from a CAT(0) cube complex: we may consider its \emph{contact graph} \cite{MR3217625}, i.e., the graph whose vertices are the hyperplanes of the cube complex and whose edges link two hyperplanes whenever their carriers intersect. So the natural question is: given a group acting on a CAT(0) cube complex, when is the induced action on the corresponding contact graph acylindrical? The question is open even when the action on the cube complex is geometric. As a partial positive answer, it is proved in \cite{BHS1} that the action on the contact graph is indeed acylindrical if the action is geometric and the cube complex \emph{hierarchically hyperbolic}, which includes cocompact special groups. As the first main theorem of this article, we prove:

\begin{thm}\label{thm:contact}
Let $G$ be a group acting on a CAT(0) cube complex $X$. If $G \curvearrowright X$ is non-uniformly weakly acylindrical, then the induced action $G \curvearrowright \Gamma X$ is non-uniformly acylindrical. 
\end{thm}

\noindent
An action by isometries $G \curvearrowright X$ is
\begin{itemize}
	\item \emph{non-uniformly weakly acylindrical} if there exists $R \geq 0$ such that, for all $x,y \in X$, $d(x,y) \geq R$ implies $| \mathrm{stab}(x) \cap \mathrm{stab}(y) | < \infty$;
	\item \emph{non-uniformly acylindrical} if, for every $d \geq 0$, there exists some constant $R \geq 0$ such that, for all $x,y \in X$, $$d(x,y) \geq R \Rightarrow \# \{ g \in G \mid d(x,gx),d(y,gy) \leq d \} < \infty.$$
\end{itemize}
The point is that, although we do not get a true acylindricity for the action on the contact graph, a non-uniformly weakly acylindrical action is still sufficient to deduce the acylindrical hyperbolicity of the group because loxodromic isometries turn out to define WPD elements; we refer to Section \ref{section:WPD} for more details. 

Interestingly, as a consequence of \cite{BBF}, it is not necessary to construct an action on a hyperbolic space to prove that a given group is acylindrically hyperbolic: it is sufficient to make a group act on an arbitrary metric space with at least one \emph{WPD contracting} isometry. We refer to Section \ref{section:WPD} for the relevant definitions. So the natural question is: given a group acting on a CAT(0) cube complex, when does it contain a WPD contracting isometry? The second main theorem of our article is a characterisation of such isometries:

\begin{thm}\label{thm:WPDmain}
Let $G$ be a group acting on a CAT(0) cube complex. Then $g \in G$ is a WPD contracting isometry if and only if it skewers a pair $(J_1,J_2)$ of well-separated hyperplanes such that the intersection $\mathrm{stab}(J_1) \cap \mathrm{stab}(J_2)$ is finite.
\end{thm}

We refer to Definition \ref{def:wellseparated} below for a precise definition of well-separated hyperplanes. According to Proposition \ref{projection ssh}, it amounts to saying that the projection of a hyperplane onto the other is bounded, and vice-versa. In order to illustrate the interest of the previous two statements, we show that each one can be used to give an alternative proof of the following criterion, proved in \cite{IndiraAlexandre}.

\begin{thm}\label{IndiraMartin}
A group acting essentially, without fixed point at infinity and non-uniformly weakly acylindrically on a finite-dimensional irreducible CAT(0) cube complex is either acylindrically hyperbolic or virtually cyclic.
\end{thm} 

Recall that an action on a CAT(0) cube complex is \emph{essential} whenever the orbit of any vertex does not lie in the neighborhood of some hyperplane, and that a CAT(0) cube complex is \emph{irreducible} if it does not split as a Cartesian product of two unbounded factors. As suggested by the work of Caprace and Sageev \cite{MR2827012}, irreducible CAT(0) cube complexes may be thought of as cube complexes with a hyperbolic behavior, justifying the previous result. 


In \cite{coningoff}, we showed that an action on a hyperbolic CAT(0) cube complex is acylindrical if and only if it is \emph{acylindrical on the hyperplanes}, i.e., if there exist $R,N \geq 1$ such that, for every hyperplanes $J_1,J_2$ separated by at least $R$ other hyperplanes, $|\mathrm{stab}(J_1) \cap \mathrm{stab}(J_2)| \leq N$. (This type of action has also been introduced independently in \cite{BeekerLazAcyl}.) Consequently, a non-virtually cyclic group admitting an action on a hyperbolic CAT(0) cube complex which is acylindrical on the hyperplanes must be acylindrically hyperbolic. As a consequence of our criteria, we are able to remove the assumption that the cube complex is hyperbolic. 

\begin{thm}\label{main1}
Let $G$ be a group acting essentially on a finite-dimensional CAT(0) cube complex. If the action is acylindrical on the hyperplanes, then $G$ contains a finite-index subgroup which is either acylindrically hyperbolic or cyclic. 
\end{thm}

We emphasize that, in the previous statement, we do not require that the cube complex is irreducible nor that the action does not fix a point at infinity. Also, notice that it is currently unknown whether being acylindrically hyperbolic is stable under commensurability (see the erratum \cite{Erratum} for more information), so Theorem \ref{main1} does not show that the group itself is acylindrically hyperbolic. Nevertheless, many of the properties proved for acylindrically hyperbolic groups are stable under commensurability, so being virtually acylindrically hyperbolic still provides non-trivial information about the group. 

An interesting criterion proved in \cite{MinasyanOsin} is that a group splitting over a malnormal subgroup must be acylindrically hyperbolic (or virtually cyclic). Because a subgroup over which a group splits must have codimension-one \cite{Scottends} (see Section \ref{section:CodimOne} for a precise definition), a natural question is:

\begin{question}
Is every non-virtually cyclic group containing a malnormal codimension-one subgroup acylindrically hyperbolic?
\end{question}

Since CAT(0) cube complexes are tightly related to codimension-one subgroups as a consequence of Sageev's work \cite{MR1347406}, we expect that they provide the good framework to deal with this question. As a consequence of Theorem \ref{main1}, we are able to give a positive answer if the subgroup is moreover finitely generated and satisfies the \emph{bounded packing property}. More generally, one shows:

\begin{prop}\label{codim}
Let $G$ be a finitely generated group. Assume that $G$ contains a finitelygenerated codimension-one subgroup which has uniformly finite height and which satisfies the bounded packing property. Then $G$ contains a finite-index subgroup which is either acylindrically hyperbolic or cyclic.
\end{prop}

\noindent
As an easy consequence of the previous proposition, one gets:

\begin{cor}\label{corcodim}
Let $G$ be a finitely generated group acting essentially on a uniformly locally finite CAT(0) cube complex $X$. If $X$ contains a hyperplane whose stabiliser is finitely generated and has uniformly finite height then $G$ is either virtually acylindrically hyperbolic or virtually cyclic.
\end{cor}


The paper is organised as follows. In Section \ref{section:preliminaries}, we begin by giving some preliminaries on CAT(0) cube complexes, needed in the sequel. Then, we study WPD contracting isometries in Section \ref{section:WPD}, proving Theorem \ref{thm:WPDmain}, and we give the first proof of Theorem \ref{IndiraMartin}. It is worth noticing that the main result of this section produces also an alternative proof of \cite[Theorem 1.1]{IndiraAlexandre}; see Remark \ref{remark:alternative}. The second proof of Theorem \ref{IndiraMartin} is given in Section 4, as a consequence of the non-uniformly acylindrical action on the contact graph provided by Theorem \ref{thm:contact}. Finally, Section \ref{section:AcylHyp} is dedicated to the proof of Theorem~\ref{main1}, and its consequences on codimension-one subgroups are studied in Section \ref{section:CodimOne}.

\section{Preliminaries}\label{section:preliminaries}

\noindent
A \textit{cube complex} is a CW complex constructed by gluing together cubes of arbitrary (finite) dimension by isometries along their faces. Furthermore, it is \textit{nonpositively curved} if the link of any of its vertices is a simplicial \textit{flag} complex (i.e., $n+1$ vertices span a $n$-simplex if and only if they are pairwise adjacent), and \textit{CAT(0)} if it is nonpositively curved and simply-connected. See \cite[page 111]{MR1744486} for more information.

A fundamental feature of cube complexes is the notion of a \textit{hyperplane}. Let $X$ be a nonpositively curved cube complex. Formally, a \textit{hyperplane} $J$ is an equivalence class of edges, where two edges $e$ and $f$ are equivalent whenever there exists a sequence of edges $e=e_0,e_1,\ldots, e_{n-1},e_n=f$ where, for every $0 \leq i \leq n-1$, the edges $e_i$ and $e_{i+1}$ are parallel sides of some square in $X$. Notice that a hyperplane is uniquely determined by one of its edges, so if $e \in J$ we say that $J$ is the \textit{hyperplane dual to $e$}. Geometrically, a hyperplane $J$ is rather thought of as the union of the \textit{midcubes} transverse to the edges belonging to $J$. We refer to this union as the \emph{geometric realisation} of $J$. See Figure \ref{figure27}. Two hyperplanes are \emph{transverse} if they are distinct and if their geometric realisations intersect. The \textit{carrier} $N(J)$ of a hyperplane $J$ is the smallest subcomplex of $X$ containing the geometric realisation of $J$, i.e., the union of the cubes intersecting the midcubes associated to $J$. In the following, $\partial N(J)$ will denote the union of the cubes of $X$ contained in $N(J)$ but not intersecting the geometric realisation of $J$, and $X \backslash \backslash J= \left( X \backslash N(J) \right) \cup \partial N(J)$. Notice that $N(J)$ and $X \backslash \backslash J$ are subcomplexes of $X$.
\begin{figure}
\begin{center}
\includegraphics[trim={0 13cm 10cm 0},clip,scale=0.4]{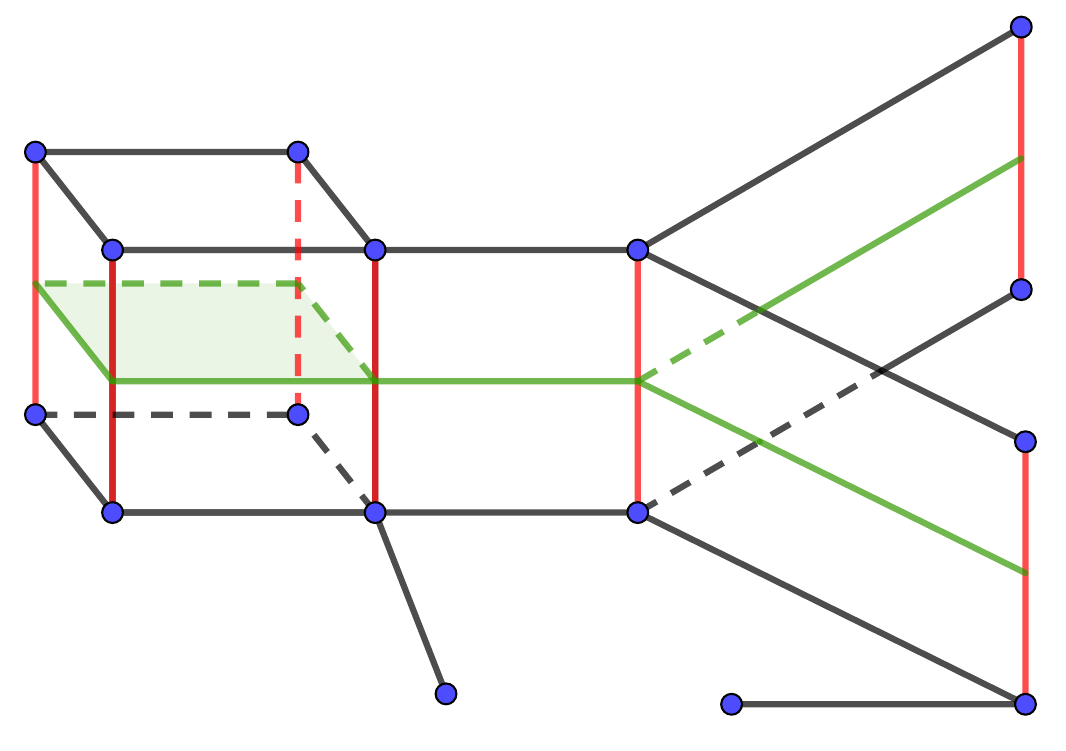}
\caption{A hyperplane and its corresponding midcubes.}
\label{figure27}
\end{center}
\end{figure}

\begin{thm}\emph{\cite[Theorem 4.10]{MR1347406}}
Let $X$ be a CAT(0) cube complex and $J$ a hyperplane. Then $X \backslash \backslash J$ has exactly two connected components.
\end{thm}

\noindent
The two connected components of $X \backslash \backslash J$ will be referred to as the \textit{halfspaces delimited} by the hyperplane $J$. One says that $J$ \emph{separates} two points of $X$, or more generally two subsets of $X$, if they lie in distinct halfspaces delimited by $J$.

\paragraph{Distances $\ell_p$.} 
There exist several natural metrics on a CAT(0) cube complex. For example, for any $p \in (0,+ \infty)$, the $\ell_p$-norm defined on each cube can be extended to a distance defined on the whole complex, the \emph{$\ell_p$-metric}. Usually, the $\ell_1$-metric is referred to as the \emph{combinatorial distance} and the $\ell_2$-metric as the \emph{CAT(0) distance}. Indeed, a CAT(0) cube complex endowed with its CAT(0) distance turns out to be a CAT(0) space \cite[Theorem C.9]{Leary}, and the combinatorial distance between two vertices corresponds to the graph metric associated to the 1-skeleton $X^{(1)}$. In particular, \textit{combinatorial geodesics} are edge-paths of minimal length, and a subcomplex is \textit{combinatorially convex} if it contains any combinatorial geodesic between two of its points.

In fact, the combinatorial metric and the hyperplanes are strongly linked together: the combinatorial distance between two vertices corresponds exactly to the number of hyperplanes separating them \cite[Theorem 2.7]{MR2413337}, and

\begin{thm}\emph{\cite[Corollary 2.16]{MR2413337}}
Let $X$ be a CAT(0) cube complex and $J$ a hyperplane. The two components of $X \backslash \backslash J$ are combinatorially convex, as well as the components of $\partial N(J)$.
\end{thm}

The $\ell_{\infty}$-metric, denoted by $d_{\infty}$, is also of particular interest. Alternatively, given a CAT(0) cube complex $X$, the distance $d_{\infty}$ between two vertices corresponds to the distance associated to the graph obtained from $X^{(1)}$ by adding an edge between two vertices whenever they belong to a common cube. Nevertheless, the distance we obtain remains strongly related to the combinatorial structure of $X$:

\begin{prop}\emph{\cite[Corollary 2.5]{depth}}
Let $X$ be a CAT(0) cube complex and $x,y \in X$ two vertices. Then $d_{\infty}(x,y)$ is the maximal number of pairwise disjoint hyperplanes separating $x$ and $y$.
\end{prop}

\paragraph{Combinatorial projection.}
In CAT(0) spaces, and so in particular in CAT(0) cube complexes with respect to the CAT(0) distance, the existence of a well-defined projection onto a given convex subspace provides a useful tool. Similarly, with respect to the combinatorial distance, it is possible to introduce a \emph{combinatorial projection} onto a combinatorially convex subcomplex, defined by the following result.

\begin{prop}\label{projection}
\emph{\cite[Lemma 1.2.3]{arXiv:1505.02053}} Let $X$ be a CAT(0) cube complex, $C \subset X$ a combinatorially convex subcomplex and $x \in X \backslash C$ a vertex. Then there exists a unique vertex $y \in C$ minimising the distance to $x$. Moreover, for any vertex of $C$, there exists a combinatorial geodesic from it to $x$ passing through $y$.
\end{prop}

\noindent
The following proposition will be especially useful:

\begin{prop}\label{hyperplanseparantcor2}\emph{\cite[Proposition 2.6]{article3}}
Let $X$ be a CAT(0) cube complex, $C$ a combinatorially convex subcomplex, $p : X \to C$ the combinatorial projection onto $C$ and $x,y \in X$ two vertices. The hyperplanes separating $p(x)$ and $p(y)$ are precisely the hyperplanes separating $x$ and $y$ which intersect $C$. \end{prop}

\noindent
As an application of this proposition, let us mention a characterisation of well-separated hyperplanes which will be useful later:

\begin{definition}\label{def:wellseparated}
Let $L \geq 0$. Two hyperplanes are \emph{$L$-well-separated} if any collection of hyperplanes intersecting both our two hyperplanes, and which does not contain any \emph{facing triple} (i.e., three hyperplanes such that each one does not separate the two others), has cardinality at most $L$. Moreover, they are 
\begin{itemize}
	\item \emph{well-separated} if they are $L$-well-separated for some $L \geq 0$;
	\item \emph{strongly separated} if they are $0$-well-separated. 
\end{itemize}
\end{definition}

\noindent
Our characterisation of well-separated hyperplanes is the following:

\begin{prop}\label{projection ssh}
Let $X$ be a CAT(0) cube complex and $J,H$ two hyperplanes of $X$. Let $p : X \to N(J)$ denote the combinatorial projection onto $N(J)$. Then $J$ and $H$ are $L$-well-separated if and only if $p(N(H))$ has diameter at most $L$.
\end{prop}

\begin{proof} Suppose first that $J$ and $H$ are $L$-well-separated. Let $x,y \in N(H)$ be two vertices and let $\mathcal{H}$ denote the set of the hyperplanes separating them. According to Proposition~\ref{hyperplanseparantcor2}, any hyperplane separating $p(x)$ and $p(y)$ separates $x$ and $y$. Therefore, $\mathcal{H}$ defines a family of hyperplanes intersecting both $J$ and $H$ which does not contain any facing triple, hence 
$$d(x,y) = \# \mathcal{H} \leq L.$$
Thus, we have proved that $p(N(H))$ has diameter at most $L$.  

\medskip \noindent
Conversely, suppose that $p(N(H))$ has diameter at most $L$, and let $\mathcal{H}$ be a finite family of hyperplanes intersecting both $J$ and $H$ which does not contain any facing triple. If $x,y \in N(H)$ are two vertices separated by each hyperplane of $\mathcal{H}$, it follows from Proposition \ref{hyperplanseparantcor2} that each hyperplane of $\mathcal{H}$ separates $p(x)$ and $p(y)$, hence
\begin{center}
$\# \mathcal{H} \leq d(p(x),p(y)) \leq \mathrm{diam} (p(N(H))) \leq L$.
\end{center}
Thus, we have proved that $J$ and $H$ are $L$-well-separated.
\end{proof}

\paragraph{Combinatorial isometries of CAT(0) cube complexes.}
Let $X$ be a CAT(0) cube complex and $g \in \mathrm{Isom}(X)$ an isometry. As a consequence of \cite[Theorem 1.4 and Lemma 4.2]{arXiv:0705.3386}, we know that exactly one the following possibilities must happen (up to subdividing $X$):
\begin{itemize}
	\item $g$ is \emph{elliptic}, i.e., $g$ stabilises a cube of $X$;
	\item $g$ is \emph{loxodromic}, i.e., there exists a bi-infinite combinatorial geodesic on which $\langle g \rangle$ acts by translations.
\end{itemize}
Naturally, if $g$ is loxodromic, we call an \emph{axis} of $g$ a bi-infinite combinatorial geodesic $\gamma$ on which $\langle g \rangle$ acts by translations. We denote by $\mathcal{H}(\gamma)$ the set of the hyperplanes intersecting $\gamma$.

\medskip \noindent
A subcomplex $Y \subset X$ is \emph{quasiconvex} if there exists some $K \geq 0$ such that, for any two vertices $x,y \in Y$, every combinatorial geodesic between $x$ and $y$ lies in the $K$-neighborhood of $Y$. We will say that an isometry is \emph{quasiconvex} if it admits a quasiconvex combinatorial axis. Recall from \cite[Proposition 3.3]{article3} that:

\begin{prop}\label{quasiconvex}\emph{\cite{article3}}
A bi-infinite combinatorial geodesic $\gamma$ is quasiconvex if and only if the joins of hyperplanes in $\mathcal{H}(\gamma)$ are uniformly thin.
\end{prop}

\noindent
A \emph{join of hyperplanes} $(\mathcal{H}, \mathcal{V})$ is the data of two collections of hyperplanes which do not contain any facing triple so that any hyperplane of $\mathcal{H}$ is transverse to any hyperplane of $\mathcal{V}$. It is \emph{$C$-thin} if $\min( \# \mathcal{H}, \# \mathcal{V}) \leq C$.

\medskip \noindent
\textbf{Convention:} In all the article, an action of a group on a metric space is always by isometries. And an action of a group on a cube complex always sends a cube to a cube of the same dimension.

\section{WPD contracting isometries}\label{section:WPD}

\noindent
If a group $G$ acts on a metric space $(S,d)$, and if $g \in G$, we say that $g$ is \emph{WPD} if, for every $d \geq 0$ and $x \in S$, there exists some $m \geq 1$ such that 
\begin{center}
$\{ h \in G \mid d(x,hx),d(g^mx,hg^mx) \leq d\}$ 
\end{center}
is finite. In \cite[Theorem 1.2]{OsinAcyl}, Osin proves that a group is acylindrically hyperbolic if and only if it is not virtually cyclic and it acts on a hyperbolic space with a WPD loxodromic isometry. This characterisation was generalised in \cite[Theorem H]{BBF} by Bestvina, Bromberg and Fujiwara as:

\begin{thm}\label{BBF}\emph{\cite{BBF}}
If a group acts on a geodesic metric space with a WPD contracting isometry, then it is either virtually cyclic or acylindrically hyperbolic.
\end{thm}

\noindent
Recall that, given a metric space $X$, an isometry $g \in \mathrm{Isom}( X)$ is \emph{contracting} if
\begin{itemize}
	\item $g$ is \emph{loxodromic}, i.e., there exists $x_0 \in X$ such that $n \mapsto g^n \cdot x_0$ defines a quasi-isometric embedding $\mathbb{Z} \to X$;
	\item the diameter of the nearest-point projection of any ball disjoint from $C_g$ onto $C_g$ is uniformly bounded, where $C_g$ is defined as $\{ g^n \cdot x_0 \mid n \in \mathbb{Z} \}$.
\end{itemize}
For instance, any loxodromic isometry of a hyperbolic space is contracting. In \cite[Theorem 3.13]{article3}, we characterised contracting isometries of CAT(0) cube complexes. In particular,

\begin{thm}\label{equivalence - contracting isometry}\emph{\cite{article3}}
An isometry of a CAT(0) cube complex is contracting if and only if it skewers a pair of well-separated hyperplanes.
\end{thm}

\noindent
Recall that an isometry $g$ \emph{skewers} a pair of hyperplanes $(J_1,J_2)$ if there exist halfspaces $J_1^+,J_2^+$ delimited by $J_1,J_2$ respectively such that $g^nJ_1^+ \subsetneq J_2^+ \subsetneq J_1^+$ for some $n \geq 1$. 

\medskip \noindent
The main result of this section is:

\begin{thm}\label{isometrie WPD}
Let $G$ be a group acting on a CAT(0) cube complex. Then $g \in G$ is a WPD contracting isometry if and only if it skewers a pair $(J_1,J_2)$ of well-separated hyperplanes such that the intersection $\mathrm{stab}(J_1) \cap \mathrm{stab}(J_2)$ is finite.
\end{thm}

\noindent
In order to prove this theorem, the following result will be needed. 

\begin{prop}\label{power}
If a group $G$ acts on a CAT(0) cube complex $X$ with a quasiconvex WPD element $g \in G$, then, for every $n \geq 1$, $g^n$ is a WPD element as well.
\end{prop}

\noindent
Before turning to the proof of Proposition \ref{power}, we need two preliminary lemmas. We begin with a probably well-known lemma; we include a proof here because no reference could be found.

\begin{lemma}\label{WPDvariation}
Let $G$ be a group acting on a metric space $(S,d)$ and $g \in G$. Then $g$ is WPD if and only if there exists some $x \in S$ such that, for every $d \geq 0$, there exists some $m \geq 1$ such that $\{ h \in G \mid d(x,hx), d(g^mx,hg^mx) \leq d \}$ is finite.
\end{lemma}

\begin{proof}
The implication is clear. Conversely, fix some $d \geq 0$ and some $y \in S$. By assumption, there exist $x \in S$ and $m \geq 1$ such that 
$$F_1:= \{ h \in G \mid d(x,hx), d(g^mx,hg^mx) \leq 2d(x,y)+d \}$$
is finite. We claim that
$$F_2:= \{ h \in G \mid d(y,hy),d(g^my,hg^my) \leq d \}$$
must be finite, showing that $g$ is indeed WPD. Fix an element $h \in F_2$. We have
$$\begin{array}{lcl} d(x,hx) &  \leq & d(x,y)+d(y,hy)+ d(hy,hx)= 2d(x,y)+d(y,hy) \\ & \leq & 2d(x,y) + d \end{array}$$
and similarly
$$\begin{array}{lcl} d(g^mx,hg^mx) & \leq & d(g^mx,g^my)+d(g^my,hg^my)+d(hg^my,hg^mx) \\ & \leq & 2d(x,y)+d(g^my,hg^my) \leq 2d(x,y)+d \end{array}$$
hence $h \in F_1$. Thus, we have proved the inclusion $F_2 \subset F_1$. As the set $F_1$ is finite, we conclude that $F_2$ has to be finite as well, as desired. 
\end{proof}

\begin{lemma}\label{lem:quadri}
Let $X$ be a CAT(0) cube complex and $\gamma$ a combinatorial geodesic between two vertices $x,y$ such that every join of hyperplanes in $\mathcal{H}(\gamma)$ are $C$-thin. If $g \in \mathrm{Isom}(X)$ satisfies $d(x,gx),d(y,gy) \leq d$, then $d(z,gz) \leq 2(C+3d)$ for every $z \in \gamma$. 
\end{lemma}

\begin{proof}
First, fix two combinatorial geodesics $[x,gx]$ and $[y,gy]$, and denote by
\begin{itemize}
	\item $\mathcal{H}_1$ the set of the hyperplanes separating $\{ gx,z \}$ and $\{gz,y\}$;
	\item $\mathcal{H}_2$ the set of the hyperplanes separating $\{ x,gz \}$ and $\{z,gy \}$;
	\item $\mathcal{H}_3$ the set of the hyperplanes separating $\{ x,gx \}$ and $\{z,gz \}$.
\end{itemize}
See Figure \ref{quadri}. Notice that $\mathcal{H}_1 \cap \mathcal{H}_3= \mathcal{H}_2 \cap \mathcal{H}_3= \emptyset$. Indeed, a hyperplane of $\mathcal{H}_3$ cannot cross $g \gamma$ (resp. $\gamma$) between $gz$ and $gy$ (resp. $z$ and $y$) as it must cross $g \gamma$ (resp. $\gamma$) at most once because $g \gamma$ (resp. $\gamma$) is a geodesic. Consequently, a hyperplane of $\mathcal{H}_3$ does not separate $gz$ and $gy$ nor $y$ and $z$, so that it cannot belong to $\mathcal{H}_1$ or $\mathcal{H}_2$.
\begin{figure}
\begin{center}
\includegraphics[scale=0.4]{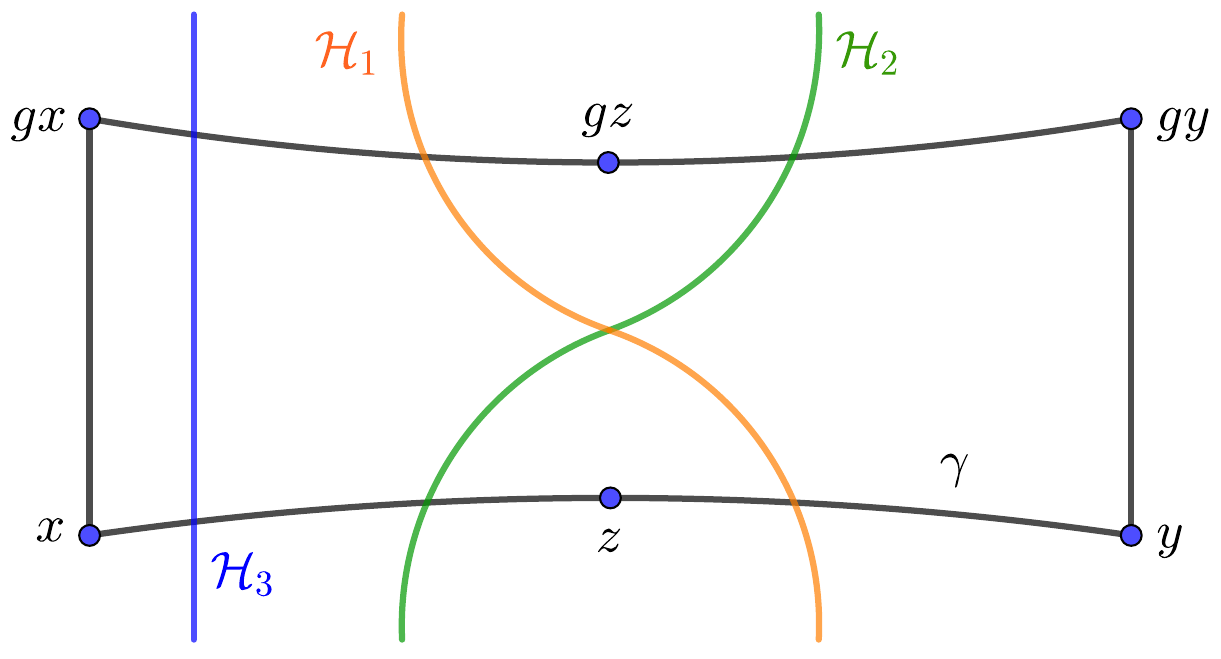}
\caption{Configuration from the proof of Lemma \ref{lem:quadri}.}
\label{quadri}
\end{center}
\end{figure}

\medskip \noindent
Now, notice that a hyperplane separating $x$ and $z$ must separate $x$ and $y$ as well because $\gamma$ is a geodesic, and so it has to cross the path $[x,gx] \cup g \gamma \cup [gy,y]$. If the hyperplane crosses $[x,gx]$, then it separates $x$ and $gx$; if it crosses $g \gamma$ between $gx$ and $gz$, then it belongs to $\mathcal{H}_3$; if it crosses $g \gamma$ between $gz$ and $gy$, then it belongs to $\mathcal{H}_2$; and if it crosses $[gy,y]$, then it separates $y$ and $gy$. Because $d(x,gx),d(y,gy) \leq d$, we deduce that 
$$0 \leq d(x,z)- \# \mathcal{H}_2 - \# \mathcal{H}_3 \leq 2d.$$
A similar argument shows that a hyperplane separating $gx$ and $gz$ either belongs to $\mathcal{H}_1$ or $\mathcal{H}_3$, or separates $x$ and $gx$ or $y$ and $gy$, hence
$$0 \leq d(gx,gz)- \# \mathcal{H}_1- \# \mathcal{H}_3  \leq 2d.$$
Consequently,
$$\begin{array}{lcl} |\# \mathcal{H}_1- \# \mathcal{H}_2| & = & | \# \mathcal{H}_1+ \# \mathcal{H}_3 - d(gx,gz)+d(x,z) - \# \mathcal{H}_2- \# \mathcal{H}_3| \\ & \leq & |d(gx,gz)- \# \mathcal{H}_1- \# \mathcal{H}_3 | + |d(x,z)- \# \mathcal{H}_2 - \# \mathcal{H}_3| \\ & \leq & 2d+2d=4d. \end{array}$$
Then, since a hyperplane separating $z$ and $gz$ either belongs to $\mathcal{H}_1$ or $\mathcal{H}_2$, or separates $x$ and $gx$ or $y$ and $gy$, we deduce that
$$d(z,gz) \leq \# \mathcal{H}_1+ \# \mathcal{H}_2+ 2d \leq 2 \min(\# \mathcal{H}_1, \# \mathcal{H}_2) + 6d.$$
Now, notice that any hyperplane of $\mathcal{H}_1$ is transverse to any hyperplane of $\mathcal{H}_2$. Indeed, let $J_1 \in \mathcal{H}_1$ and $J_2 \in \mathcal{H}_2$ be two hyperplanes. Because $J_1$ may intersect $\gamma$ at most once (as $\gamma$ is a geodesic), we know that $J_1$ does not cross $\gamma$ between $x$ and $z$. Consequently, $J_1$ separates $\{x,z\}$ and $\{ y \}$. Similarly, it separates $\{gx\}$ and $\{gz,gy\}$. But $J_2$ separates $gz$ and $gy$, and $x$ and $z$, so that $J_2$ has to intersect the two halfspaces delimited by $J_1$. Therefore, $J_1$ and $J_2$ must be transverse.

\medskip \noindent
In other words, $(\mathcal{H}_1, \mathcal{H}_2)$ defines a join of hyperplanes in $\mathcal{H}(\gamma)$, which has to be $C$-thin, i.e., $\min(\# \mathcal{H}_1, \# \mathcal{H}_2) \leq C$. We conclude that $d(z,gz) \leq 2C+6d$ as desired.
\end{proof}

\begin{proof}[Proof of Proposition \ref{power}.]
Let $\gamma$ be a quasiconvex combinatorial axis for $g$; according to Proposition \ref{quasiconvex}, we know that there exists some $C \geq 1$ such that any join of hyperplanes in $\mathcal{H}(\gamma)$ is $C$-thin. We fix some vertex $x \in \gamma$ and some $d \geq 0$. Because $g$ is WPD, there exists some $m \geq 1$ such that
\begin{center}
$\{ h \in G \mid d(x,hx),d(g^mx,hg^mx) \leq 2(C+3d) \}$
\end{center}
is finite. Now, let $h \in G$ satisfy $d(x,hx),d(g^{mn}x,hg^{mn}x) \leq d$. Because $g^mx$ is a vertex of $\gamma$ between $x$ and $g^{mn}x$, it follows from the previous lemma that $d(g^mx,hg^mx) \leq 2(C+3d)$. We conclude that
\begin{center}
$\{ h \in G \mid d(x,hx),d(g^{mn}x,hg^{mn}x) \leq d \}$
\end{center}
is finite, so that $g^{n}$ is WPD according to Lemma \ref{WPDvariation}.
\end{proof}

\begin{proof}[Proof of Theorem \ref{isometrie WPD}.]
First, suppose that $g$ skewers a pair $(J_1,J_2)$ of well-separated hyperplanes such that the intersection $\mathrm{stab}(J_1) \cap \mathrm{stab}(J_2)$ is finite. We already know that $g$ is contracting thanks to Theorem \ref{equivalence - contracting isometry}. It remains to show that it is also WPD. Fix some $d \geq 0$. 

\medskip \noindent
If $J_1^+, J_2^+$ are halfspaces delimited by $J_1,J_2$ respectively such that $g^n J_1^+ \subsetneq J_2^+ \subsetneq J_1^+$ for some $n \geq 1$, then we also have $g^{kn} \subsetneq J_2^+ \subsetneq J_1^+$ for every $k \geq 1$, so we can choose $n$ sufficiently large so that $n > 2d+1$. Moreover, notice that $H=\mathrm{stab}(J_1) \cap \mathrm{stab}(g^nJ_1)$ is finite. Indeed, because there exist only finitely many hyperplanes separating $J_1$ and $g^nJ_1$, $H$ contains a finite-index subgroup $H_0$ stabilising each of these hyperplanes; since $J_2$ separates $J_1$ and $g^nJ_1$, we deduce that $H_0$ is a subgroup of $\mathrm{stab}(J_1) \cap \mathrm{stab}(J_2)$, which is finite. A fortiori, $H$ must be finite.

\medskip \noindent
Therefore, if we fix a combinatorial axis $\gamma$ of $g$, there exists a hyperplane $J \in \mathcal{H}(\gamma)$ and an integer $n > 2d+1$ such that $J$ and $g^nJ$ are disjoint and $\mathrm{stab}(J) \cap \mathrm{stab}(g^nJ)$ is finite. Fix some vertex $x \in \gamma \cap N(J)$. If we set
$$F= \{ h \in G \mid d(x,hx),d(g^{n+2d}x,hg^{n+2d}x) \leq d\},$$
according to Lemma \ref{WPDvariation}, it is sufficient to prove that $F$ is finite in order to conclude that $g$ is a WPD element of $G$. 

\medskip \noindent
For convenience, let $\mathcal{H}$ denote the set of the hyperplanes separating $x$ and $g^{n+2d}x$. We claim that an element of $F$ sends all but at most $2d$ elements of $\mathcal{H}$ into $\mathcal{H}$. So let $f \in F$ be an element and $H \in \mathcal{H}$ a hyperplane. Because $H$ separates $x$ and $g^{n+2d}x$, necessarily $fH$ separates $fx$ and $fg^{n+2d}x$. Now, if $fH$ does not separate $x$ and $g^{n+2d}x$, necessarily $fH$ must separate either $x$ and $fx$ or $g^{n+2d}x$ and $fg^{n+2d}x$. But we know that $d(x,hx), \ d(g^{n+2d}x,hg^{n+2d}x) \leq d$, so at most $2d$ such hyperplanes may exist, proving our claim.

\medskip \noindent
Thus, if $\mathcal{W}$ denotes the set $\{ g^k J \mid 0 \leq k \leq n+2d \}$ and $\mathcal{L}$ the set of functions $(S \subset \mathcal{W}) \to \mathcal{H}$, where $S$ has cocardinality at most $2d$ in $\mathcal{W}$ (i.e., $|\mathcal{W} \backslash S| \leq 2d$), then any element of $F$ induces an element of $\mathcal{L}$. If $F$ is infinite, then because $\mathcal{L}$ is finite there must exist infinitely many pairwise distinct elements $g_0,g_1,g_2, \ldots \in F$ inducing the same function of $\mathcal{L}$. In particular, $g_0^{-1}g_1,g_0^{-1}g_2, \ldots$ stabilise each hyperplane of a subset $S \subset \mathcal{W}$ of cocardinality at most $2d$. We deduce that there exists some $0 \leq k \leq 2d$ such that 
$$\mathrm{stab}(g^kJ) \cap \mathrm{stab}(g^{k+n}J)= \left( \mathrm{stab}(J) \cap \mathrm{stab}(g^nJ) \right)^{g^k}$$
is infinite, which contradicts our assumption.

\medskip \noindent
Conversely, suppose that $g$ is a WPD contracting isometry. According to Theorem \ref{equivalence - contracting isometry}, $g$ skewers a pair of well-separated hyperplanes $(J_1,J_2)$, i.e., there exists some $n \geq 1$ such that $g^n J_1^+ \subsetneq J_2^+ \subsetneq J_1^+$ for some halfspaces $J_1^+,J_2^+$ delimited by $J_1,J_2$ respectively. Notice that, because $J_2$ separates $J_1$ and $g^nJ_1$ and because $J_1$ and $J_2$ are well-separated, necessarily $J_1$ and $g^nJ_1$ are well-separated. So there exist some hyperplane $J$ and some constant $n \geq 1$ such that $J$ and $g^nJ$ are $L$-well-separated and $g^nJ^+ \subsetneq J^+$ for some halfspace $J^+$ delimited by $J$. Let $C$ denote the combinatorial projection of $N(g^nJ)$ onto $N(J)$. According to Proposition \ref{projection ssh}, $C$ has diameter at most $L$.

\medskip \noindent
Now, fix some vertex $z \in C$. Because $g^n$ is WPD as well, since any contracting contracting isometry is quasiconvex (see for example \cite[Lemma 2.20]{article3}) so that Proposition~\ref{power} applies, there exists some $m \geq 1$ such that
\begin{center}
$\{ h \in G \mid d(z,hz),d(g^{nm}z,hg^{nm}z) \leq L \}$
\end{center}
is finite. First, we want to prove that $H= \bigcap\limits_{i=0}^{m+1} \mathrm{stab}(g^{ni}J)$ is finite. Because $H$ stabilises $J$ and $g^nJ$, necessarily $H$ stabilises $C$, whose diameter is at most $L$; similarly, because $H$ stabilises $g^{nm}J$ and $g^{n(m+1)}J$, $H$ must stabilise $g^{nm}C$, which is the combinatorial projection of $N(g^{n(m+1)}J)$ onto $N(g^{nm}J)$, and so has diameter at most $L$. Therefore, $z \in C$ and $g^{nm} z \in g^{nm}C$ implies 
\begin{center}
$d(z,hz) \leq L$ and $d(g^{nm}z,hg^{nm}z) \leq L$
\end{center}
for every $h \in H$. We conclude that $H$ is finite. On the other hand, because there exist only finitely many hyperplanes separating $J$ and $g^{n(m+1)}J$, we know that $H$ is a finite-index subgroup of $\mathrm{stab}(J) \cap \mathrm{stab}(g^{n(m+1)}J)$. Therefore, $\mathrm{stab}(J) \cap \mathrm{stab}(g^{n(m+1)}J)$ has to be finite as well. 

\medskip \noindent
Finally, notice that $g^{n(m+1)}J^+ \subsetneq g^nJ^+ \subsetneq J^+$. As a consequence, $g^nJ$ separates $g^{n(m+1)}J$ and $J$. Because $J$ and $g^nJ$ are well-separated, it follows that $J$ and $g^{n(m+1)}J$ are also well-separated. Thus, we have proved that $g$ skewers the pair $(J,g^{n(m+1)}J)$ of well-separated hyperplanes where $\mathrm{stab}(J) \cap \mathrm{stab} (g^{n(m+1)}J)$ is finite.
\end{proof}

\begin{remark}\label{remark:alternative}
Theorem \ref{isometrie WPD} provides also an alternative proof of \cite[Theorem 1.1]{IndiraAlexandre}, which states that if a group $G$ acts essentially without fixed point at infinity on an irreducible finite dimensional CAT(0) cube complex such that there exist two hyperplanes whose stabilisers intersect along a finite subgroup, then $G$ must be acylindrically hyperbolic or virtually cyclic. The beginning of the argument remains unchanged: finding two strongly separated hyperplanes $J_1,J_2$ such that $\mathrm{stab}(J_1) \cap \mathrm{stab}(J_2)$ is finite. Next, instead of constructing an \"{u}ber-contraction, we deduce from \cite[Double Skewering Lemma]{MR2827012} that there exists an isometry $g \in G$ skewering the pair $(J_1,J_2)$. According to Theorem \ref{isometrie WPD}, $g$ turns out to be a WPD contracting isometry, so that $G$ must be acylindrically hyperbolic or virtually cyclic as a consequence of Theorem \ref{BBF}.
\end{remark}

\begin{proof}[First proof of Theorem \ref{IndiraMartin}.]
According to \cite[Proposition 5.1]{MR2827012}, $X$ contains a pair $(J_1,J_2)$ of strongly separated hyperplanes. Then, it follows from \cite[Double Skewering Lemma]{MR2827012} that there exists an element $g \in G$ skewering this pair. This proves that there exist a hyperplane $J$ and an integer $n \geq 1$ such that $J$ and $g^{kn}J$ are strongly separated for every $k \geq 1$. If we prove that $\mathrm{stab}(J) \cap \mathrm{stab}(g^{kn}J)$ is finite for some $k \geq 1$, then we will be able to conclude that $g$ is a WPD contracting isometry according to Theorem \ref{isometrie WPD}, so that the conclusion will follow from Theorem \ref{BBF}.

\medskip \noindent
Since the combinatorial projections of $N(J)$ onto $N(g^{kn}J)$ and of $N(g^{kn}J)$ onto $N(J)$ are two single vertices according to Proposition \ref{projection ssh}, we deduce that $\mathrm{stab}(J) \cap \mathrm{stab}(g^{kn}J)$ fixes two vertices $x \in N(J)$ and $x_k \in N(g^{kn}J)$. If we choose $k$ sufficiently large so that the distance between $x$ and $x_k$ turns out to be sufficiently large, we deduce from the non-uniform weak acylindricity of the action that $\mathrm{stab}(J) \cap \mathrm{stab}(g^{kn}J)$ is finite.
\end{proof}

\section{Action on the contact graph}

\noindent
In \cite{MR3217625}, Hagen associated to any CAT(0) cube complex $X$ a hyperbolic graph, namely the \emph{contact graph} $\Gamma X$. This is the graph whose vertices are the hyperplanes of $X$ and whose edges link two hyperplanes $J_1,J_2$ whenever $N(J_1) \cap N(J_2) \neq \emptyset$. In \cite[Corollary 14.5]{BHS1}, Behrstock, Hagen and Sisto proved that, if a group acts geometrically on a CAT(0) cube complex which admits an invariant factor system, then the induced action on the contact graph is acylindrical. The question of whether this action is acylindrical without additional assumption on the cube complex remains open. The main result of this section suggests a positive answer. It is worth noticing that, although we are not able to deduce a complete acylindricity of the action on the contact graph, no assumption is made on the cube complex and the action of our group is not supposed to be geometric but only to satisfy some weak acylindrical condition.

\begin{thm}\label{acylindricite}
Let $G$ be a group acting on a CAT(0) cube complex $X$. If the action of $G$ on $X$ is non-uniformly weakly acylindrical, then the induced action $G \curvearrowright \Gamma X$ is non-uniformly acylindrical. As a consequence, with respect to the action $G \curvearrowright \Gamma X$, any loxodromic isometry of $G$ turns out to be WPD.
\end{thm}

\noindent
Given two hyperplanes $J$ and $H$, let $\Delta(J,H)$ denote the maximal length of a \emph{chain} of pairwise strongly separated hyperplanes $V_1, \ldots, V_n$ separating $J$ and $H$, i.e., for every $2 \leq i \leq n-1$ the hyperplane $V_i$ separates $V_{i-1}$ and $V_{i+1}$. Our first result estimates the distance in $\Gamma$ thanks to $\Delta(\cdot,\cdot)$. 

\begin{prop}\label{qi}
Let $X$ be a CAT(0) cube complex. For every pair of hyperplanes $J$ and $H$, we have 
$$\Delta(J,H) \leq d_{\Gamma X}(J,H) \leq 5 \Delta(J,H)+6.$$
\end{prop}

\noindent
Our proposition will be a direct consequence of Lemmas~\ref{qi1} and~\ref{qi2} proved below. The following result, essentially contained in \cite[Chapter 3]{Hagenthesis}, will be needed to prove these lemmas.

\begin{lemma}\label{Hagen}
Let $V_1, \ldots, V_n$ be the successive vertices of a geodesic in $\Gamma X$ and let $H_1, \ldots, H_m$ denote the hyperplanes separating $V_1$ and $V_n$ in $X$. 
\begin{itemize}
	\item[(i)] For every $2 \leq i \leq n-1$, there exists some $1 \leq j \leq m$ such that $d_{\Gamma X}(V_i,H_j) \leq 1$. 
	\item[(ii)] For every $1 \leq j \leq m$, there exists some $2 \leq i \leq n-1$ such that $d_{\Gamma X}(V_i,H_j) \leq 1$.
\end{itemize}
\end{lemma}

\begin{proof}
The assertion (i) is exactly \cite[Lemma 4.2]{MR3217625}. Next, fix some $1 \leq j \leq m$. Because $H_j$ separates $V_1$ and $V_n$, and because $N(V_1) \cup \cdots \cup N(V_n)$ is connected, necessarily there exists some $1 \leq i \leq n$ such that $H_j$ crosses $N(V_i)$, i.e., $H_j$ coincides with $V_i$ or is transverse to it. In fact, as $H_j$ separates $V_1$ and $V_n$, the index $i$ cannot be $1$ nor $n$, so $2 \leq i \leq n-1$. We have $d_{\Gamma X} (V_i,H_j) \leq 1$, proving (ii).
\end{proof}

\noindent
Our two lemmas are the followings:

\begin{lemma}\label{qi1}
If $J,H$ are two hyperplanes satisfying $d_{\Gamma X}(J,H) \geq 5n+1$, then there exist at least $n$ pairwise strongly separated hyperplanes separating $J$ and $H$ in $X$. 
\end{lemma}

\begin{proof}
Let $J=V_0,V_1, \ldots, V_{r-1},V_r=H$ be a geodesic in $\Gamma X$ between $J$ and $H$. According to Lemma \ref{Hagen}, for every $1 \leq k \leq r-1$, there exists a hyperplane $S_k$ separating $J$ and $H$ such that $d_{\Gamma X}(V_k,S_k) \leq 1$. For every $1 \leq k \leq (r-1)/5$ and every $1 \leq j \leq (r-1)/5-k$, we have
\begin{center}
$\begin{array}{lcl} d_{\Gamma X}(S_{5k},S_{5(k+j)}) & \geq & d_{\Gamma X}(V_{5k},V_{5(k+j)})-d_{\Gamma X}(V_{5k},S_{5k})-d_{\Gamma X}(V_{5(k+j)},S_{5(k+j)}) \\ \\ & \geq & 5j -1 -1 \geq 3 \end{array}$
\end{center}
Therefore, $S_{5k}$ and $S_{5(k+j)}$ are strongly separated. 
\end{proof}

\begin{lemma}\label{qi2}
Let $J$ and $H$ be two hyperplanes. If they are separated in $X$ by $n$ pairwise strongly separated hyperplanes $V_1, \ldots, V_n$ such that $V_i$ separates $V_{i-1}$ and $V_{i+1}$ in $X$ for every $2 \leq i \leq n-1$, then $d_{\Gamma X}(J,H) \geq n$. 
\end{lemma}

\begin{proof}
Let $J=S_0,S_1, \ldots, S_{r-1},S_r=H$ be a geodesic in $\Gamma X$ between $J$ and $H$. According to Lemma \ref{Hagen}, for every $1 \leq k \leq n$, there exists some $1 \leq n_k \leq r-1$ such that $d_{\Gamma X}(V_k,S_{n_k}) \leq 1$. Notice that, for every $1 \leq i< j \leq n$, because $V_i$ and $V_j$ are strongly separated, necessarily $n_i \neq n_j$. Let $\varphi$ be a permutation so that the sequence $(n_{\varphi(k)})$ is increasing. We have
\begin{center}
$\begin{array}{lcl} d_{\Gamma X}(J,H)& = & \displaystyle \sum\limits_{k=1}^n d_{\Gamma X}(S_{n_{\varphi(k)}}, S_{n_{\varphi(k+1)}}) \\ \\  & \geq & \displaystyle \sum\limits_{k=1}^n \left( d_{\Gamma X}(V_{\varphi(k)},V_{\varphi(k+1)}) - d(V_{\varphi(k)}, S_{n_{\varphi(k)}}) -d(V_{\varphi(k+1)}, S_{n_{\varphi(k+1)}}) \right) \\ \\ & \geq & \displaystyle \sum\limits_{k=1}^n (3-1-1) = n, \end{array}$
\end{center}
where we used the inequality $d_{\Gamma X}(V_{\varphi(k)},V_{\varphi(k+1)}) \geq 3$, which precisely means that $V_{\varphi(k)}$ and $V_{\varphi(k+1)}$ are strongly separated. This completes the proof. 
\end{proof}

\begin{proof}[Proof of Proposition \ref{qi}.]
Let $J$ and $H$ be two hyperplanes. Because 
$$d_{\Gamma X}(J,H) \geq 5 \cdot \left\lfloor \frac{d_{\Gamma X}(J,H)-1}{5} \right\rfloor +1,$$
it follows from Lemma~\ref{qi1} that there exists a chain of at least $\lfloor (d_{\Gamma X}(J,H)-1)/5 \rfloor$ pairwise strongly separated hyperplanes separating $J$ and $H$, hence
$$\Delta(J,H) \geq \left\lfloor \frac{d_{\Gamma X}(J,H)-1}{5} \right\rfloor \geq \frac{d_{\Gamma X}(J,H)-1}{5} -1,$$ 
and finally $d_{\Gamma X}(J,H) \leq 5 \Delta(J,H)+6$. Next, the inequality $\Delta(J,H) \leq d_{\Gamma X } (J,H)$ is a direct consequence of Lemma \ref{qi2}. 
\end{proof}

\noindent
We are now ready to prove Theorem \ref{acylindricite}.

\begin{proof}[Proof of Theorem \ref{acylindricite}.] Let $R_0$ be the constant given by the non-uniform weak acylindricity of the action $G \curvearrowright X$. Let $\epsilon>0$ and $R \geq 5(R_0+4( \epsilon +4 \delta)+6)+6$ where $\delta$ is the hyperbolicity constant of $\Gamma X$. Now, fix two hyperplanes $J$ and $H$ satisfying $d_{\Gamma X}(J,H) \geq R$ and let
$$F = \{ g \in G \mid d_{\Gamma X}(J,gJ) , d_{\Gamma X}(H,gH) \leq \epsilon \}.$$
Our goal is to show that $F$ is finite.

\medskip \noindent
According to Proposition \ref{qi}, there exist $m \geq R_0+4( \epsilon +4 \delta)+6$ pairwise strongly separated hyperplanes $V_1, \ldots, V_m$ separating $J$ and $H$ and such that $V_i$ separates $V_{i-1}$ and $V_{i+1}$ for every $2 \leq i \leq m-1$. Because $m$ is large enough, there exist integers $1 \leq r < p < q < s \leq m$ such that
$$\left\{ \begin{array}{l} |r-p|, |q-s| \geq \epsilon+2+8 \delta \\ r, |m-s| > \epsilon \\ |p-q| \geq R_0 \end{array} \right.$$
The point is that we have controlled lower bounds on the numbers of pairwise strongly separated hyperplanes separating two consecutive hyperplanes in the chain $V_1$, $V_r$, $V_p$, $V_q$, $V_s$, $V_m$.

\medskip \noindent
We claim that, for every $g \in F$ and every hyperplane $W$ separating $V_p$ and $V_q$, $gW$ separates $V_r$ and $V_s$. 

\medskip \noindent
First of all, we notice that 
$$d_{\Gamma X}(W,gW) \leq \epsilon+8 \delta+2.$$ 
Indeed, if $S_0=J, S_1, \ldots, S_{r-1},S_r=H$ is a geodesic in $\Gamma X$ between $J$ and $H$, according to Proposition \ref{qi}, there exists $1 \leq j \leq r-1$ such that $d_{\Gamma X}(W,S_j) \leq 1$. Next, notice that $d_{\Gamma X}(J,H)= d_{\Gamma X} (gJ,gH)$, $d_{\Gamma X}(J,S_j) = d_{\Gamma X} (gJ,gS_j)$ and $d_{\Gamma X}(J,gJ),d_{\Gamma X}(H,gH) \leq \epsilon$ by our hypotheses; see Figure \ref{convex}. As a consequence of \cite[Corollaire 10.5.3]{CDP}, we have $d_{\Gamma X}(S_j,gS_j) \leq \epsilon+ 8\delta$. Therefore,
$$d_{\Gamma X}(W,gW) \leq d_{\Gamma X}(W,S_j)+ d_{\Gamma X}(S_j,gS_j)+ d_{\Gamma X}(gS_j,gW) \leq \epsilon+8 \delta +2,$$
as desired.
\begin{figure}
\begin{center}
\includegraphics[scale=0.4]{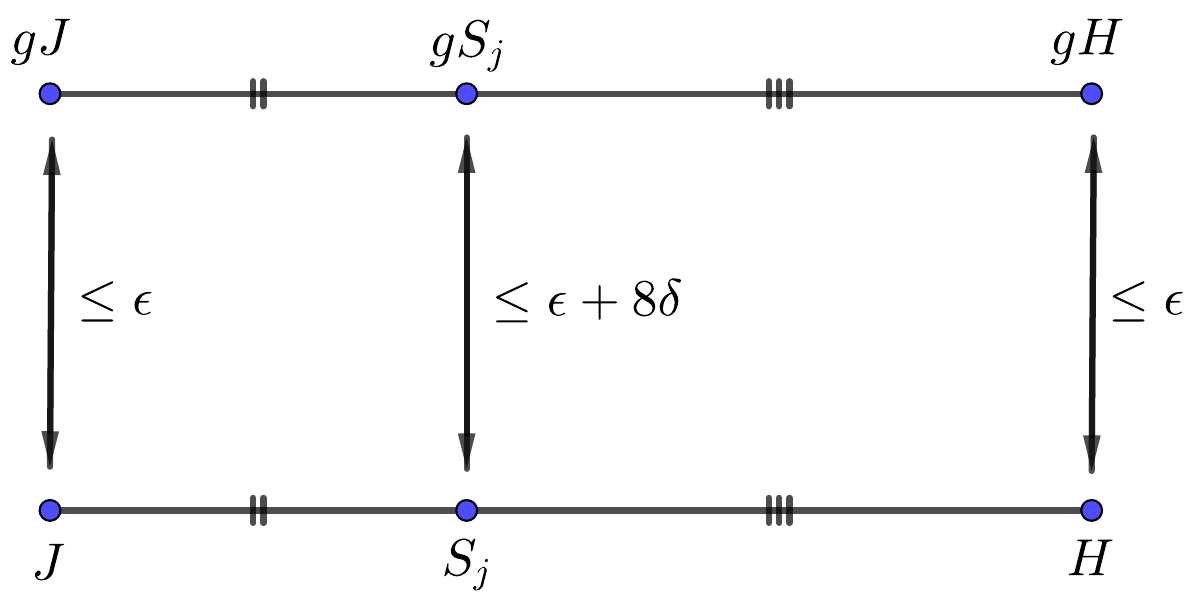}
\caption{Configuration from the proof of Theorem \ref{acylindricite}.}
\label{convex}
\end{center}
\end{figure}

\medskip \noindent
For convenience, for every $1 \leq i \leq m$, let $V_i^+$ (resp. $V_i^-$) denote the halfspace delimited by $V_i$ containing $H$ (resp. $J$). If $gW \subset V_{r+1}^{-}$, then $V_{r+1}, \ldots, V_p$ separate $W$ and $gW$, hence $d_{\Gamma X}(W,gW) > |r-p| \geq \epsilon+2+8\delta$ according to Proposition \ref{qi}, which contradicts our previous observation. Therefore, $gW$ intersects $V_{r+1}^+$. As $V_r$ and $V_{r+1}$ are strongly separated, $gW$ cannot be transverse to both $V_r$ and $V_{r+1}$, which implies that $gW$ cannot intersect $V_r^-$, hence $gW \subset V_r^+$. Similarly, we prove that $gW \subset V_s^-$. Thus, we have proved that $gW$ lies between $V_r$ and $V_s$, i.e., $gW \subset V_r^+ \cap V_s^-$.

\medskip \noindent
Now, if $gW$ does not separate $V_r$ and $V_s$, then $gW^+ \subset V_r^+ \cap V_s^-$ for some halfspace $W^+$ delimited by $W$. In particular, because $W^+$ contains $J$ or $H$, we deduce that either $V_1, \ldots, V_r$ separate $J$ and $gJ$ or $V_s, \ldots, V_m$ separate $H$ and $gH$. Thus, it follows from Proposition \ref{qi} that either $d_{\Gamma X}(J,gJ) \geq r > \epsilon$ or $d_{\Gamma X}(H,gH) \geq |m-s|> \epsilon$, which is a contradiction. Therefore, $gW$ has to separate $V_r$ and $V_s$, concluding the proof of our claim.

\medskip \noindent
Let $\mathcal{H}(a,b)$ denote the (finite) set of the hyperplanes separating $V_a$ and $V_b$ for every $1 \leq a,b \leq m$, and let $\mathcal{L}$ denote the set of functions $\mathcal{H}(p,q) \to \mathcal{H}(r,s)$. We have proved that any element of $F$ induces a function of $\mathcal{L}$. If $F$ is infinite, there must exist $g_0,g_1,g_2, \ldots \in F$ inducing the same function of $\mathcal{L}$. As a consequence, the elements $g_0^{-1}g_1, g_0^{-1}g_2, \ldots$ belong to $I:= \bigcap\limits_{W \in \mathcal{H}(p,q)} \mathrm{stab}(W)$. But, any element of $I$ stabilises $V_{p+1}$ and $V_{q-1}$, and a fortiori the combinatorial projections of $N(V_{p+1})$ onto $N(V_{q-1})$ and of $N(V_{q-1})$ onto $N(V_{p+1})$. So, by applying Proposition~\ref{projection ssh}, we find two vertices $x \in N(V_{p+1})$ and $y \in N(V_{q-1})$ fixed by $I$. Now, $x$ and $y$ are separated by $V_{p+2}, \ldots, V_{q-2}$, hence $d(x,y) \geq |p-q| +4 > R_0$. The non-uniform weak acylindricity of the action $G \curvearrowright X$ implies that $I$ must be finite, a contradiction. Therefore, $F$ is necessarily finite. We conclude that the induced action $G \curvearrowright \Gamma X$ is non-uniform acylindrical.
\end{proof}

\begin{proof}[Second proof of Theorem \ref{IndiraMartin}.]
According to \cite[Proposition 5.1]{MR2827012}, $X$ contains a pair $(J_1,J_2)$ of strongly separated hyperplanes. Then, it follows from \cite[Double Skewering Lemma]{MR2827012} that there exists an element $g \in G$ skewering this pair. On the other hand, we know that an isometry which skewers a pair of strongly separated hyperplanes induces a loxodromic isometry on the contact graph $\Gamma X$. This is essentially a consequence of Proposition \ref{qi}; otherwise see \cite[Theorem 6.1.1]{Hagenthesis}. Furthermore, $g$ is WPD according to Theorem \ref{acylindricite}. Therefore, we have proved that $G$ acts on a hyperbolic space with a WPD isometry. The conclusion follows.
\end{proof}

\section{Acylindrical actions on the hyperplanes}\label{section:AcylHyp}

\noindent
By analogy with non-uniformly acylindrical actions, let us define an action of a group $G$ on a CAT(0) cube complex $X$ to be \emph{non-uniformly acylindrical on the hyperplanes} if there exists some $R \geq 1$ such that, for every hyperplanes $J_1,J_2$ separated by at least $R$ other hyperplanes, the intersection $\mathrm{stab}(J_1) \cap \mathrm{stab}(J_2)$ is finite. We begin by noticing that a non-uniformly acylindrical action on the hyperplanes is non-uniformly weakly acylindrical, so that Theorem \ref{IndiraMartin} applies under the former hypothesis.

\begin{lemma}\label{impl}
Let $G$ be a group acting on a finite-dimensional CAT(0) cube complex $X$. If the action is non-uniformly acylindrical on the hyperplanes, then it is non-uniformly weakly acylindrical.
\end{lemma}

\begin{proof}
Let $R \geq 1$ be such that, for any hyperplanes $J_1$ and $J_2$ separated by at least $R$ hyperplanes, $\mathrm{stab}(J_1) \cap \mathrm{stab}(J_2)$ is finite. Without loss of generality, we may suppose that $R \geq \dim (X)$. Now, let $L$ be the Ramsey number $\mathrm{Ram}(R+2)$; by definition, if you color the vertices of the complete graph containing at least $L$ vertices, so that two adjacent vertices have different colors, then there exists a monochromatic set with at least $R+2$ vertices. Notice that, if $x,y \in X$ are two vertices satisfying $d(x,y) \geq L$, then, by definition of $L$ and because $X$ does not contain $R+2$ pairwise transverse hyperplanes, there must exist $R+2$ pairwise non-transverse hyperplanes separating $x$ and $y$. Consequently, if $H$ denotes the finite-index subgroup of $\mathrm{stab}(x) \cap \mathrm{stab}(y)$ which stabilises each hyperplane separating $x$ and $y$, then $H$ stabilises two hyperplanes which are separated by at least $R$ other hyperplanes. We conclude that $H$, and a fortiori $\mathrm{stab}(x) \cap \mathrm{stab}(y)$, is finite.
\end{proof}

\noindent
In the proof of Theorem \ref{main1} below, we use \emph{restriction quotients} of CAT(0) cube complexes. Given a CAT(0) cube complex $X$ and a collection of hyperplanes $\mathcal{H}$, the \emph{restriction quotient} $X(\mathcal{H})$ is the CAT(0) cube complex obtained by cubulating the space with walls $(X, \mathcal{H})$. We refer to \cite{MR2827012} for more information. 

\begin{proof}[Proof of Theorem \ref{main1}.]
According to \cite[Proposition 2.6]{MR2827012}, it is possible to decompose the collection $\mathcal{H}$ of all the hyperplanes of $X$ as a disjoint union $\mathcal{H}_1 \sqcup \cdots \sqcup \mathcal{H}_r$ such that $X$ is isomorphic to the Cartesian product of the restriction quotients $X(\mathcal{H}_1) \times \cdots \times X(\mathcal{H}_r)$, where each factor is irreducible, and $G$ contains a finite-index subgroup $\dot{G}$ such that this decomposition is $\dot{G}$-invariant. Notice that, for each $1 \leq i \leq r$, the induced action $\dot{G} \curvearrowright X(\mathcal{H}_i)$ is again essential and acylindrical on the hyperplanes, because these properties are preserved under taking a restriction quotient (see \cite[Proposition 3.2]{MR2827012} for essential actions). Therefore, $G$ contains a finite-index subgroup $\dot{G}$ acting essentially and acylindrically on the hyperplanes on a finite-dimensional irreducible CAT(0) cube complex $Y$. 

\medskip \noindent
If $\dot{G} \curvearrowright Y$ has no fixed point at infinity, then Theorem \ref{IndiraMartin} (which applies thanks to Lemma \ref{impl}) implies that $\dot{G}$ is acylindrically hyperbolic or virtually cyclic. Otherwise, \cite[Proposition 2.26]{CIF} implies that two cases may happen: either $\dot{G}$ has a finite orbit in the Roller boundary of $Y$, or $\dot{G}$ contains a finite-index subgroup $\ddot{G}$ and $Y$ admits a $\ddot{G}$-invariant restriction quotient $Z$ such that the induced action $\ddot{G} \curvearrowright Z$ has no fixed point at infinity. In the latter case, we can apply Theorem \ref{IndiraMartin} (thanks to Lemma \ref{impl}) to deduce that $\ddot{G}$ is either acylindrically hyperbolic or virtually cyclic.

\medskip \noindent
From now on, suppose that the action $\dot{G} \curvearrowright Y$ has a finite orbit in the Roller boundary of $Y$. In particular, $\dot{G}$ contains a finite-index subgroup $\ddot{G}$ fixing an ultrafilter $\alpha$ in the Roller boundary of $Y$. It follows from \cite[Theorem B.1]{CIF} that $\ddot{G}$ contains a normal subgroup $F$, which is \emph{locally elliptic} in the sense that any finitely generated subgroup of $F$ fixes a point of $Y$, such that the quotient $ \ddot{G} /F$ is a finitely generated free abelian group. 

\begin{claim}\label{locally finite}
There exists a constant $K$ such that any elliptic subgroup $H$ of $\ddot{G}$ (i.e., any subgroup fixing a point of $Y$) has cardinality at most $K$. 
\end{claim}

\noindent
Because the action $\ddot{G} \curvearrowright Y$ is acylindrical on the hyperplanes, there exist two constants $L,N$ such that, for any hyperplane $J_1$ and $J_2$ of $Y$ separated by at least $L$ hyperplanes, the intersection $\mathrm{stab}(J_1) \cap \mathrm{stab}(J_2)$ has cardinality at most $N$. Let $x \in Y$ be a point fixed by $H$. Let us denote
$$U(x, \alpha)= \{  J \ \text{hyperplane} \mid \text{$x$ and $\alpha$ lie in distinct halfspaces delimited by $J$} \}.$$
Alternatively, $U(x,\alpha)$ can be interpreted as the set of the hyperplanes intersecting a combinatorial ray starting from $x$ and pointing to $\alpha$. Because $H$ fixes $x$ and $\alpha$, $U(x,\alpha)$ is $H$-invariant. If $\mathcal{H}_i$ denotes the set of the hyperplanes $J$ of $U(x,\alpha)$ satisfying $d_{\infty}(x,N(J))=i$, then
$$U(x, \alpha) = \mathcal{H}_0 \sqcup \mathcal{H}_1 \sqcup \mathcal{H}_2 \sqcup \cdots.$$
Notice that, because $H$ fixes $x$ and stabilises $U(x, \alpha)$, each $\mathcal{H}_i$ is $H$-invariant. On the other hand, according to \cite[Claim 4.9]{article3}, for each $i \geq 0$, $\mathcal{H}_i$ is a collection of pairwise transverse hyperplanes, hence $\# \mathcal{H}_i \leq \dim(X)$. Consequently, if $H$ has cardinality at least $(N+2) \cdot (\dim(X)!)^2$, then there exist at least $N +2$ pairwise distinct elements $h_0, \ldots, h_{N+1} \in H$ which induce the same permutation on $\mathcal{H}_0 \cup \mathcal{H}_{L+1}$; in particular, $h_0^{-1}h_1, \ldots, h_0^{-1}h_{N+1}$ define $N+1$ pairwise distinct elements fixing the hyperplanes of $\mathcal{H}_0 \cup \mathcal{H}_{L+1}$. But, according to \cite[Claim 4.10]{article3}, there exist two hyperplanes $J_1 \in \mathcal{H}_0$ and $J_2 \in \mathcal{H}_{L+1}$ separated by at least $L$ hyperplanes, so we have proved that $\mathrm{stab}(J_1) \cap \mathrm{stab}(J_2)$ has cardinality at least $N+1$, contradicting the acylindricity on the hyperplanes. Therefore, $H$ has cardinality at most $(N+2) \cdot (\dim(X)!)^2$, proving our claim.

\medskip \noindent
Suppose now by contradiction that the subgroup $F \lhd \ddot{G}$ is infinite. In particular, $F$ contains an infinite countable subgroup $C = \{ g_1, g_2, \ldots \}$. Indeed, either $F$ is countable and we can take $C=F$; or $F$ is uncountable, and because it cannot be finitely generated, there exist elements $g_1,g_2, \ldots$ such that $g_i \notin \langle g_1, \ldots, g_{i-1} \rangle$ for every $i \geq 2$, so that we can take $C= \{g_1,g_2, \ldots, \}$. For every $i \geq 1$, let $C_i$ denote the subgroup $\langle g_1, \ldots, g_i \rangle$. We have
\begin{center}
$C_1 \subset C_2 \subset C_3 \subset \cdots \subset \bigcup\limits_{i \geq 1} C_i = C$.
\end{center}
Because $F$ is locally elliptic, each $C_i$ has to be elliptic and the previous claim implies that its cardinality is bounded above by a constant which does not depend on $i$. Therefore, the sequence $C_1 \subset C_2 \subset \cdots$ must be eventually constant, and in particular $C$ is necessarily finite, a contradiction. We conclude that $F$ is finite.

\medskip \noindent
So we know that there exists an exact sequence $1 \to F \to \ddot{G} \to \mathbb{Z}^k \to 1$ for some $k \geq 0$, where $F$ is finite. If $k \leq 1$, then $\ddot{G}$ is either finite or virtually infinite cyclic. From now on, suppose that $k \geq 2$. Let $a,b \in \ddot{G}$ be the lifts of two independent elements of $\mathbb{Z}^k$. In particular, for every $p \geq 1$, the commutator $[a,b^p]$ belongs to the subgroup $F$; because $F$ is finite, we deduce that there exist two integers $n \neq m$ such that $[a,b^n]=[a,b^m]$, which implies $[a,b^{m-n}]=1$. Let $A$ denote the subgroup of $\ddot{G}$ generated by $a$ and $b^{m-n}$: this is a free abelian subgroup of rank two. 

\medskip \noindent
We know from Claim \ref{locally finite} that the induced action $A \curvearrowright Y$ has no global fixed point, so, by a result proved by Sageev (see \cite[Proposition B.8]{CIF}), there exist a hyperplane $J$ of $Y$ and an element $g \in A$ such that $gJ^+ \subsetneq J^+$ for some halfspace $J^+$ delimited by $J$. In particular, for every $n \geq 1$, the hyperplanes $J$ and $g^n J$ are separated by $n-1$ hyperplanes, namely $g J, g^2 J, \ldots, g^{n-1} J$. Therefore, because the action $A \curvearrowright Y$ is also acylindrical on the hyperplanes, for some sufficiently large $n \geq 1$, the intersection $\mathrm{stab}_A(J) \cap \mathrm{stab}_A(g^nJ)$ is finite, and a fortiori trivial since $A$ is torsion-free. Thus, because $A$ is abelian,
\begin{center}
$\mathrm{stab}_A(J)= \mathrm{stab}_A(J) \cap \mathrm{stab}_A(J)^{g^n} = \mathrm{stab}_A(J) \cap \mathrm{stab}_A(g^nJ)=\{1\}$.
\end{center}
It follows from \cite[\S 3.3]{MR1347406} that $\{1 \}$ is a codimension-one subgroup of $A$ (see the next section for a precise definition), or equivalently, that $A$ has at least two ends. We get a contradiction, since we know that $A$ is a free abelian group of rank two.

\medskip \noindent
We conclude that $G$ is necessarily virtually cyclic.
\end{proof}

\begin{remark}
In the statement of Theorem \ref{main1}, \emph{acylindrical on the hyperplanes} cannot be replaced with \emph{non-uniformly acylindrical on the hyperplanes} (where you do not require the intersection between the stabilisers of two sufficiently far away hyperplanes to be uniformly bounded, but only finite). Indeed, it is a consequence of the construction given in the proof of \cite[Theorem I.6.15]{MR1954121} that any countable locally finite group acts essentially and non-uniformly acylindrically on the hyperplanes on a tree. Explicitly, let $G$ be a countable locally finite group. Because it is countable, it can be written as a union of finitely generated subgroups $G_1 \subset G_2 \subset \cdots$, which are finite since $G$ is locally finite; for a specific example, you may consider
\begin{center}
$\mathbb{Z}_2 \subset \mathbb{Z}_2 \oplus \mathbb{Z}_2 \subset \mathbb{Z}_2 \oplus \mathbb{Z}_2 \oplus \mathbb{Z}_2 \subset \cdots \subset \bigoplus\limits_{i \geq 1} \mathbb{Z}_2$.
\end{center}
Now, let $T$ be the graph whose vertices are the cosets $gG_n$, where $g \in G$ and $n \geq 1$, and whose edges link two cosets $gG_n$ and $hG_{n+1}$ if their images in $G/G_{n+1}$ coincide. This is a tree, the action is essential and non-uniformly acylindrical on the hyperplanes. Notice that the action is non-uniformly weakly acylindrical (but not weakly acylindrical). This does not contradict Theorem \ref{IndiraMartin} because $G$ fixes the point at infinity corresponding to the geodesic ray $(G_1,G_2,G_3, \ldots)$. 
\end{remark}

\begin{question}
Let $G$ be a group acting essentially on a finite-dimensional CAT(0) cube complex. If the action is non-uniformly acylindrical on the hyperplanes, is $G$ either acylindrically hyperbolic or (locally finite)-by-cyclic?
\end{question}

\section{An application to codimension-one subgroups}\label{section:CodimOne}

\noindent
Given a finitely generated group $G$ and a subgroup $H \leq G$, we can define the \emph{relative number of ends} $e(G,H)$ as the number of ends of the quotient of a Cayley graph of $G$ (with respect to a finite generating set) by the action of $H$ by left-multiplication; this definition does not depend on the choice of the generating set. See \cite{Scottends} for more information. If $e(G,H) \geq 2$, we say that $H$ is a \emph{codimension-one subgroup}. The main result of \cite{MR1347406} states that to any codimension-one subgroup is associated an essential action of $G$ on a CAT(0) cube complex, transitive on the hyperplanes, so that the hyperplane-stabilisers contain a conjugate of our codimension-one subgroup as a finite-index subgroup. Unfortunately, this complex may be infinite-dimensional. Nevertheless, the CAT(0) cube complex we construct turns out to be finite-dimensional if our codimension-one subgroup $H\leq G$ is finitely generated and satisfies the \emph{bounded packing property}: fixing a Cayley graph of $G$ (with respect to a finite generating set), for every $D \geq 1$, there is a number $N$ so that, for any collection of $N$ distinct cosets of $H$ in $G$, at least two are separated by a distance at least $D$. Although not stated explicitly, this idea goes back to \cite{MR1438181}; see \cite{packing} for more information. Therefore,

\begin{thm}\label{cubulation}
Let $G$ be a finitely generated group and $H$ a finitely generated codimension-one subgroup. Suppose that $H$ satisfies the bounded packing property. Then $G$ acts essentially on a finite-dimensional CAT(0) cube complex, transitively on the hyperplanes, so that its hyperplane-stabilisers contain a conjugate of $H$ as a finite-index subgroup.
\end{thm}

\noindent
As mentioned above, this theorem is essentially contained into \cite{MR1347406, MR1438181}. As some of the properties we need are not explicitly stated there (but follow easily from the constructions involved), we include a sketch of proof of the reader's convenience.

\begin{proof}[Sketch of proof of Theorem \ref{cubulation}.]
Let $A_0$ be an invariant $H$-set (as defined in \cite[Section 2.2]{MR1347406}) and set $\Sigma= \{ gA_0 \mid g \in G \} \cup \{ gA_0^c \mid g \in G\}$. Let $X$ denote the CAT(0) cube complex constructed in \cite[Section 3.1]{MR1347406}. The vertices of $X$ are subsets $W$ of $\Sigma$ satisfying the following two conditions:
\begin{itemize}
	\item for every $A\in \Sigma$, $W$ contains exactly one element among $\{A,A^c\}$;
	\item for every $A,B \in \Sigma$, if $A \subset B$ and $A \in W$ then $B \in W$;
\end{itemize}
and the edges of $X$ link two vertices $W_1$ and $W_2$ if their symmetric difference is $2$, or in other words, if there exists $A \in \Sigma$ such that $W_2= (W_1 \backslash \{A\}) \cup \{A^c\}$ and $W_1 = (W_2 \backslash \{A^c \} ) \cup \{A\}$. Because $G$ naturally acts on $\Sigma$ by left-multiplication, it also acts on the vertices of $X$, and this action extends to an essential action of $G$ on $X$ according to \cite[Theorem 3.10]{MR1347406}. 

\medskip \noindent
By construction, the edges of $X$ are naturally labelled by pairs $\{A,A^c\}$ where $A \in \Sigma$. According to \cite[Lemma 3.9]{MR1347406}, for every hyperplane $J$ of $X$, there exists some $A \in \Sigma$ such that the edges dual to $J$ are exactly those labelled by $\{A,A^c\}$. As a consequence, the map $J \mapsto \{A,A^c\}$ defines a $G$-equivariant bijection from the collection of the hyperplanes of $X$ and the set of pairs $\{A,A^c\}$ where $A \in \Sigma$. It follows that $G$ acts on $X$ with a single orbit of hyperplanes, and that hyperplane-stabilisers have the form $\mathrm{stab}(\{A,A^c\})$ where $A \in \Sigma$. According to \cite[Lemma 2.4]{MR1347406}, $H$ has finite index in $\mathrm{stab}(A)$, so we deduce that each hyperplane-stabiliser contains a conjugate of $H$ as a finite-index subgroup. 

\medskip \noindent
Finally, the fact that $X$ is finite-dimensional is a consequence of \cite[Corollary~3.31]{packing}.
\end{proof}

\noindent
The combination of the previous result with Theorem \ref{main1} essentially proves Proposition \ref{codim}. First, recall a subgroup $H \leq G$ has \emph{(uniformly) finite height} if there exists some $n \geq 1$ such that, for every collection of $n$ distinct cosets $g_1H, \ldots, g_nH$, the cardinality of the intersection $\bigcap\limits_{i=1}^n g_iHg_i^{-1}$ is (uniformly) finite.

\begin{proof}[Proof of Proposition \ref{codim}.]
It follows from Theorem \ref{cubulation} that $G$ acts essentially on a finite-dimensional CAT(0) cube complex $X$, transitively on the hyperplanes, so that its hyperplane-stabilisers contain a conjugate of $H$ as a finite-index subgroup. Since there is only one orbit of hyperplanes, there exists some $M \geq 0$ such that any hyperplane-stabiliser contains a conjugate of $H$ as subgroup of index at most $M$. For convenience, fix a hyperplane $A$ of $X$ whose stabiliser contains $H$ as a subgroup of index at most $M$. Since $H$ has uniformly finite height, we can fix some $n,N \geq 1$ such that, for every collection of $n$ distinct cosets $g_1H, \ldots, g_nH$, the intersection $\bigcap\limits_{i=1}^n g_iHg_i^{-1}$ has cardinality at most $N$. 

\medskip \noindent
Let $J_1,J_2$ be two hyperplanes separated by at least $n \cdot \dim(X)$ other hyperplanes. Let $\mathcal{H}$ be the set of the hyperplanes separating $J_1$ and $J_2$. For $i \geq 0$, let $\mathcal{H}_i$ denote the set of the hyperplanes $J$ of $\mathcal{H}$ satisfying $d_{\infty}(N(J),N(J_1))=i$, i.e., the maximal number of pairwise disjoint hyperplanes separating $J_1$ and $J$ is $i$; it is not difficult to notice that the hyperplanes of $\mathcal{H}_i$ are pairwise transverse, hence $\# \mathcal{H}_i \leq \dim(X)$ (see for instance the proof of \cite[Claim 4.9]{article3}). In particular, because $J_1$ and $J_2$ are separated by at least $n \cdot \dim(X)$ hyperplanes, we know that $\mathcal{H}_0, \ldots, \mathcal{H}_{n-1}$ are non-empty. Because $K=\mathrm{stab}(J_1) \cap \mathrm{stab}(J_2)$ preserves $\mathcal{H}_i$ for every $i \geq 0$, it must contain a subgroup $K_0$ of index at most $(\dim(X)!)^n$ which stabilises each hyperplane of $\mathcal{H}_0, \ldots, \mathcal{H}_{n-1}$.

\medskip \noindent
Now, for every $0 \leq i \leq n-1$, fix a hyperplane $H_i \in \mathcal{H}_i$; because $G$ acts on $X$ with a single orbit of hyperplanes, there must exist some $g_i \in G$ such that $H_i=g_iA$. We claim that the cosets $g_0H, \ldots, g_{n-1}H$ are pairwise distinct. Indeed, let $0 \leq i,j \leq n-1$ be two indices such that $g_iH=g_jH$. So there exists some $h \in H$ such that $g_j=g_ih$. As $H$ stabilises the hyperplane $A$, it follows that
$$H_j=g_jA=g_ihA= g_iA=H_i.$$
But the collections $\mathcal{H}_0, \ldots, \mathcal{H}_{n-1}$ are pairwise disjoint, so the hyperplanes $H_0, \ldots, H_{n-1}$ must be pairwise distinct. This implies that $i=j$, proving our claim. As a consequence, the intersection $I:=\bigcap\limits_{i=0}^{n-1} g_iHg_i^{-1}$ has cardinality at most $N$.

\medskip \noindent
Notice that, according to the group theoretical lemma below, $I$ has index at most $M!^n$ in $\bigcap\limits_{i=0}^{n-1} \mathrm{stab}(H_i)$, hence
$$\begin{array}{lcl} \# K & \leq & \dim(X)!^n \cdot \# K_0 \leq \dim(X)!^n \cdot \left| \bigcap\limits_{i=0}^{n-1} \mathrm{stab}(H_i) \right| \\ \\ & \leq & (M! \cdot \dim(X)!)^n \# I \leq (M! \cdot \dim(X)!)^n \cdot N. \end{array}$$
Thus, we have proved that, for any pair of hyperplanes separated by at least $n \cdot \dim(X)$ other hyperplanes, the intersection of their stabilisers has cardinality at most $(M! \cdot \dim(X)!)^n \cdot N$. Consequently, the action $G \curvearrowright X$ is acylindrical on the hyperplanes, and it follows from Theorem~\ref{main1} that $G$ is either acylindrically hyperbolic or virtually cyclic.
\end{proof}

\begin{lemma}
Let $G$ be a group and $G_1,H_1, \ldots, G_k,H_k \leq G$ a collection of subgroups. Suppose that, for every $1 \leq i \leq k$, $H_i$ is a subgroup of $G_i$ of index at most $n_i$. Then $\bigcap\limits_{i=1}^k H_i$ has index at most $n_1! \cdots n_k!$ in $\bigcap\limits_{i=1}^k G_i$.
\end{lemma}

\begin{proof}
Let $\bigcap\limits_{i=1}^k G_i$ act by diagonal left-multiplication on $G_1/H_1 \times \cdots \times G_k / H_k$. This defines a homomorphism $$\varphi : \bigcap\limits_{i=1}^k G_i \to \mathrm{Sym}(G_1/H_1) \times \cdots \times \mathrm{Sym}(G_k/H_k)$$ 
where $\mathrm{Sym}(\cdot)$ denotes the symmetric group of the corresponding set. We clearly have $\mathrm{ker}(\varphi) \subset \bigcap\limits_{i=1}^k H_i$, hence $$\left| \bigcap\limits_{i=1}^k G_i / \bigcap\limits_{i=1}^k H_i \right| \leq \left| \bigcap\limits_{i=1}^k G_i / \mathrm{ker}(\varphi) \right| \leq \left| \mathrm{Sym}(G_1/H_1) \times \cdots \times \mathrm{Sym}( G_k/H_k) \right| = n_1! \cdots n_k!$$
proving our lemma.
\end{proof}

\begin{proof}[Proof of Corollary \ref{corcodim}.]
Let $J$ be a hyperplane whose stabiliser is finitely generated and has uniformly finite height. Because the action is essential, $H= \mathrm{stab}(J)$ is a codimension-one subgroup according to \cite[\S 3.3]{MR1347406}, and we know that $H$ satisfies the bounded packing property according to \cite[Theorem 3.2]{packing}. Thus, the conclusion follows directly from Proposition \ref{codim}.
\end{proof}

\begin{question}
If the subgroup of Proposition \ref{codim} does not satisfy the bounded packing property, does the conclusion still hold?
\end{question}

\addcontentsline{toc}{section}{References}

\bibliographystyle{alpha}
{\footnotesize\bibliography{AcylAHypCCC}}

\end{document}